\documentclass[runningheads, envcountsame, a4paper]{llncs}

\usepackage{amssymb}
\usepackage{graphicx}
\usepackage{enumerate}
\usepackage{amsmath,amsopn,amsfonts}
\usepackage[hyphens]{url}
\usepackage{microtype}

\usepackage{url}
\urldef{\mailsa}\path| berdinsky@gmail.com | 

\newcommand{\keywords}[1]{\par\addvspace\baselineskip
\noindent\keywordname\enspace\ignorespaces#1}

\begin{document}

\mainmatter  

\title{Cayley Automatic Groups and Numerical Characteristics of Turing Transducers}

\titlerunning{Cayley Automatic Groups and Num. Characteristics of Turing Transducers}

%
%
\author{Dmitry Berdinsky 
} 

%
\authorrunning{D. Berdinsky}


\institute{Department of Computer Science, The University of Auckland,\\
Private Bag 92019, Auckland 1142, New Zealand\\
\mailsa\\
}

%
%

\toctitle{Cayley Automatic Groups and Numerical Characteristics of Turing Transducers}
\tocauthor{Dmitry~Berdinsky}
\maketitle
\setcounter{footnote}{0}

\begin{abstract}      


 This paper is devoted to the problem of finding characterizations for Cayley automatic groups. 
 The concept of Cayley automatic groups was  recently introduced by Kharlampovich, Khoussainov and
 Miasnikov. 
 We address this problem by introducing three numerical characteristics of Turing transducers: 
 growth functions, F{\o}lner functions and average length growth functions. 
 These three numerical characteristics are the analogs of 
 growth functions, F{\o}lner functions and 
 drifts of simple random walks
 for Cayley graphs of groups. 
 We study these numerical characteristics for Turing transducers obtained 
 from automatic presentations of labeled directed graphs.


\end{abstract}

\section{Introduction}
 
  This paper contributes to the field of automatic structures 
	\cite{BlumensathGradel04,BlumensathGradel00,Hodgson83,khoussainov2007three,KhoussainovNerode95,khoussainov2008open} 
	with particular emphasis on 
  Cayley automatic groups \cite{KKM11}. 
  Recall that a finitely generated group $G$ is called Cayley automatic if  
  for some set of generators $S$ the labeled directed Cayley graph 
  $\Gamma(G,S)$ is an automatic structure (or, FA--presentable).   
  All automatic groups in the sense of Thurston are Cayley automatic.   
  However, the class of Cayley automatic groups is considerably wider than 
  the class of automatic groups. 
  For example, all finitely generated 
  nilpotent groups of nilpotency 
  class at most two and all fundamental 
  groups of three--dimensional manifolds are Cayley automatic \cite{KKM11}.  
  The Baumslag--Solitar groups are Cayley automatic \cite{dlt14}.     
  Cayley  automatic groups retain the key algorithmic 
  properties which hold for automatic groups: 
  the word problem for Cayley automatic groups 
  is decidable in quadratic time, the conjugacy problem for 
  Cayley biautomatic groups is decidable, and   
  the first order theory for Cayley graphs of 
  Cayley automatic groups is decidable.

  Oliver and Thomas found a characterization of FA--presentable
  groups by showing that a finitely generated group 
  is FA--presentable if and only if it is virtually abelian 
  \cite{OliverThomas05}. Their result is based partly on the 
  celebrated Gromov's theorem on groups of polynomial growth. 
  But, the problem of finding characterizations for Cayley automatic 
  groups is more complicated, and it seems to require new approaches. 
	In this paper we address this problem by introducing some  numerical characteristics  
  for Turing transducers of the special class $\mathcal{T}$.

     In Section \ref{characterizationsection} we define the class 
     of Turing transducers $\mathcal{T}$. Then we show that 
     automatic presentations of Cayley graphs of  groups can be expressed 
		 in terms of Turing transducers of the class $\mathcal{T}$. 
     This explains why study of admissible asymptotic behavior 
     for some numerical characteristics of Turing transducers of the class $\mathcal{T}$
     is relevant to the problem of finding characterizations for Cayley automatic groups. 
     In Section \ref{numcharacterTuringtransducers} we introduce 
     three numerical characteristics 
     for Turing transducers of the class $\mathcal{T}$. 
     In this paper,  wreath products of groups are  
     used as the source of examples of Cayley automatic groups. 
     Therefore, in Section \ref{wreathproducts} we briefly recall 
		 basic  definitions for wreath products of groups.  
     In Section \ref{sectionExamples} we discuss asymptotic behavior 
		 of the numerical characteristics  of Turing  transducers of the class $\mathcal{T}$. 
     Section \ref{discussionSection} concludes the paper.

  \section{Turing Transducers of the Class $\mathcal{T}$
  	and Automatic Presentations of Labeled Directed Graphs}   
  \label{characterizationsection}
  
  Recall that a $(k+1)$--tape Turing  transducer $T$ for $k\geqslant 1$ 
	is a multi--tape Turing machine which has one input tape and $k$ output tapes.
  See, e.g., \cite[\S~10]{Medunabook} 
  for the definition of Turing transducers. 
  The  special class of 
  Turing transducers $\mathcal{T}$ that we consider in this paper is described as follows.     
  Let us be given a $(k+1)$--tape Turing transducer 
  $T \in \mathcal{T}$ and an input word $x \in \Sigma^*$. 
  Initially, the input word $x$ appears on the input tape, 
  the output tapes are completely blank and all heads 
  are over the leftmost cells.   
  First the heads of $T$ move synchronously from the left to the right 
	until the end of the input $x$. Then the heads make 
  a finite number of steps (probably no steps) further to the right, where this number of steps
  is bounded from above by some constant which depends on $T$.  
  After that, the heads of $T$ move synchronously from the 
  right to the left until it enters a final state with all heads over the leftmost cells.

  We say that $T$ accepts $x$ if $T$ enters an accepting state; 
  otherwise, $T$ rejects $x$. 
  Let $L \subseteq \Sigma^*$ be the set of inputs  accepted by $T$.  
  We say that $T$ translates 
  $x \in L$ into the outputs $y_1, \dots, y_k$ if for the word
  $x$ fed to $T$ as an input,  $T$ returns the 
  word $y_i$ on the $i$th output tape of $T$ for every $i=1,\dots,k$. 
  It is assumed that for every input $x \in L$, the output 
  $y_i \in L$ for every $i =1,\dots,k$. 
  Let $L' \subseteq L^k$ be the set
  of all $k$--tuples of outputs $(y_1, \dots,y_k)$.
  We say that $T$ translates $L$
  into $L'$.     
  
  For a given finite alphabet $\Sigma$   
  put $\Sigma_\diamond = \Sigma \cup \{\diamond \}$, where 
  $\diamond \notin \Sigma$. The convolution of 
  $n$ words $w_1,\dots, w_n \ \in \Sigma^*$ is the string 
  $w_1 \otimes \dots \otimes w_n$		      
  of length $\max\{|w_1|,\dots,|w_n|\}$ over the alphabet 
  $\Sigma_\diamond ^n$ defined as follows. 
  The $k$th symbol
  of the string is $(\sigma_1,\dots,\sigma_n )$, 
  where $\sigma_i$, $i=1,\dots,n$ is the $k$th symbol of 
  $w_i$ if $k \leqslant |w_i|$ and
  $\diamond$ otherwise. 
  The convolution $\otimes R$  of a $n$--ary relation 
  $R \subseteq \Sigma^{*n} $ is defined as
  $\otimes R = \{w_1 \otimes \dots \otimes w_n | 
  (w_1, \dots, w_n) \in R\}$.
  Recall that a $n$--tape synchronous finite automaton is
  a finite automaton over the alphabet 
  $\Sigma_\diamond  ^n \setminus \{(\diamond,\dots,\diamond)\}$. 
  Let $T \in \mathcal{T}$.
  Lemma \ref{fromtranstoautomataprop1} below shows 
  connection between Turing transducers of the class $\mathcal{T}$
	and multi--tape synchronous finite automata.  
  \begin{lemma}
  	\label{fromtranstoautomataprop1} 	
  	There exists a $(k+1)$--tape synchronous finite automaton 
		$\mathcal{M}$  such that a  convolution 
  	$x \otimes y_1 \otimes \dots \otimes y_k 
  	\in \Sigma_\diamond ^{(k+1)*}$ is accepted by $\mathcal{M}$
  	iff $T$ translates the input $x$ into the outputs $y_1,\dots,y_k$.      
  	In particular, the language $L$ is regular. 
  \end{lemma}
  \begin{proof}
  	The lemma can be obtained straightforwardly from the following two well known facts. 
  	The first fact is that the class of regular languages 
  	is closed under reverse. The second fact is as follows. 
  	Let the convolutions $\otimes R_1$ and  $\otimes R_2$ of two 
  	relations $R_1 = \{ (x, y) | x,y \in \Sigma^* \}$ and 
  	$R_2 = \{(y, z) | y, z \in \Sigma^* \}$ 
  	be accepted by two--tape synchronous finite automata. 
  	Then the convolution $\otimes R$ of the relation 
  	$R = \{(x,z) | \exists y [(x,y) \in R_1 \wedge (y,z) \in R_2 ] \}$   
  	is accepted by a two--tape synchronous finite automaton. $\Box$
  \end{proof}	
	 	  
  For a given $k$, put $\Sigma_k = \{1, \dots, k\}$. 
  Let $T \in \mathcal{T}$ be a $(k+1)$--tape Turing transducer 
  translating a language $L$ into $L' \subseteq L^k $.   
  We construct the labeled directed graph $\Gamma_T$ with the
  labels from $\Sigma_k$ as follows.  
  The set of vertices $V (\Gamma_T)$ is identified with $L$.  
  For given $u,v \in L$ there is an oriented edge $(u,v)$ 
  labeled by $j \in \Sigma_k$ if 
  $T$ translates $u$ into some outputs $w_1,\dots, w_k $ such that
  $w_j = v$. 
  It is easy to see that each vertex of
  the graph $\Gamma_T$ has $k$ outgoing edges 
  labeled by $1, \dots, k$.    
  
  Let $\Gamma$ be a labeled directed graph 
  for which every vertex has $k$ outgoing edges 
  labeled by  $1, \dots, k$.
  Recall that  $\Gamma$ is 
  called automatic if there exists a bijection between a regular 
  language and the set of vertices $V(\Gamma)$ such that for every 
  $j \in \Sigma_k$ the set of oriented edges labeled by $j$ is 
  accepted by a synchronous two--tape finite automaton. 
  From Lemma \ref{fromtranstoautomataprop1} we obtain that
  $\Gamma_T$ is automatic.     
  Suppose that $\Gamma$ is automatic. 
  Lemma \ref{fromautomatatotransprop1} below shows 
  that $\Gamma$ can be obtained as $\Gamma_T$ for some 
  $(k+1)$--tape Turing transducer $T \in \mathcal{T}$.
  \begin{lemma}
  	\label{fromautomatatotransprop1}   
  	There exists a $(k+1)$--tape Turing transducer 
  	$T \in \mathcal{T}$ for which $\Gamma_T \cong \Gamma$.    
  \end{lemma}
  \begin{proof}
  	The lemma can be obtained from the following fact. 
  	Let $R = \{(x,y) | x,y \in L\}$ be a binary relation 
  	such that $\otimes R$ is recognized by a two--tape synchronous 
  	finite automaton, where $L$ is a regular language. 
  	Suppose that for every $x \in L$ there exists exactly one
  	$y \in L$ such that $(x,y) \in R$. Then there exists a 
  	two--tape Turing transducer $T_R \in \mathcal{T}$ 
  	for which  
  	$T_R$ translates $x$ into $y$ iff $(x,y) \in R$
  	and $T_R$ rejects $x$ iff $x \notin L$.  
  	The construction of the Turing transducer $T_R$ can be 
  	found, e.g., in \cite[Theorem 2.3.10]{Epsteinbook}.  
  	The resulting $(k+1)$--tape Turing transducer 
  	$T \in \mathcal{T}$ is obtained 
  	as the combination of $k$ two--tape Turing transducers
  	$T_{R_1}, \dots , T_{R_k}$, where $R_1,\dots, R_k$ are 
  	the binary relations defined by the directed
  	edges of $\Gamma$ labeled by
  	$1,\dots,k$, respectively. $\Box$   
  \end{proof} 
  Lemmas \ref{fromtranstoautomataprop1} and \ref{fromautomatatotransprop1}
  together imply the following theorem.   
  \begin{theorem}
  	\label{automatatotranstheorem} 	
  	The labeled directed graph $\Gamma$ is automatic iff 
  	there exists a Turing transducer $T \in \mathcal{T}$ for which
  	$\Gamma \cong \Gamma_T$. $\Box$
  \end{theorem}    
  Let $\Gamma (G,S)$ be a Cayley graph for some set of generators $S = \{s_1,\dots,s_k\}$. 
	Let us fix an order of elements in $S$ as $s_1,\dots,s_k$.      
  We say that the Cayley graph $\Gamma (G,S)$ is presented 
  by  $T \in \mathcal{T}$ if, 
  after changing labels from $j$ to $s_j$ for 
  every $j \in \Sigma_k$ 
  in $\Gamma_T$,  $\Gamma_T \cong \Gamma (G,S)$.
  The isomorphism $\Gamma_T \cong \Gamma (G,S)$
  defines the bijection $\psi : L \rightarrow G$ 
  up to the choice of the word of $L$ corresponding to 
  the identity $e \in G$.
  By Theorem \ref{automatatotranstheorem} we obtain that 
  if $\Gamma(G,S)$ is presented by $T \in \mathcal{T}$, then 
  $G$ is a Cayley automatic group and $T$ provides an 
  automatic presentation for the Cayley graph $\Gamma(G,S)$.
  Moreover, for each automatic
  presentation of $\Gamma(G,S)$ there is a corresponding 
  Turing transducer $T \in \mathcal{T}$ for which $\Gamma(G,S)$ is 
  presented by $T$.   
  
  \begin{remark} 
  	 Let 
  	 $\widehat{\mathcal{T}}$
  	 be the class of 
  	 Turing transducers 
  	 for which the heads 
  	 move synchronously 
  	 and that halt in 
  	 linear time 
  	 for every input. 
  	 Clearly, 
  	 $\mathcal{T} \subseteq 
  	  \widehat{\mathcal{T}}$. 
  	 However, for each
  	 Turing transducer in 
  	 the class 
  	 $\widehat{\mathcal{T}}$ 
  	 there exists a 
  	 Turing transducer in the
  	 class $\mathcal{T}$ 
  	 which accepts  the same language of inputs and for every 
  	 accepted input produces 
  	 the same output. 
  	 The latter follows 
  	 from the fact that   	  
  	 every function computed 
  	 on a (position--faithful)
  	 one--tape Turing  
  	 machine in linear time 
  	 is automatic (i.e., can be recognized by 
  	 a two--tape synchronous 
  	 finite automaton)  
  	 \cite{Stephan_lmcs_13}.
  	 Therefore, 
  	 $\widehat{\mathcal{T}}$ 
  	 and $\mathcal{T}$ present exactly the same 
  	 class of labeled directed
  	 graphs. 
   \end{remark}	
 
 \section{Numerical Characteristics of Turing Transducers}
 \label{numcharacterTuringtransducers}
  
   We now introduce three numerical characteristics for Turing transducers
   of the class $\mathcal{T}$. 
   Let $T \in \mathcal{T}$ be a $(k+1)$--tape Turing transducer 
   translating a language $L$ into $L' \subseteq L^k$. 
   Given a  word $w \in L$,
   feed $w$ to $T$. Let 
   $w_1, \dots, w_k \in L$ be the outputs of $T$ for $w$. 
   We denote by $T(w)$ the set 
   $T(w) = \{w_1,\dots,w_k\}$.   
   Given a set $W \subseteq L$, we denote by 
   $T(W)$ the set $T(W) = \bigcup\limits_{w \in W} T (w)$.  
   Let us choose a word $w_0 \in L$. Put $W_0 = \{w_0\}$, 
   $W_1 = T(W_0)$ and, for $i > 1$,  put $W_{i+1} = T(W_i)$. 
   Let $V_n = \bigcup\limits_{i=0}^{n} W_i$, $n \geqslant 0$.   
   Put $b_n  = \# V_n$. 
   \begin{itemize}
   	\item{We call the sequence $b_n, n = 0, \dots, \infty$ the 
   		growth function of  the  pair $(T,w_0)$.}    
   \end{itemize}  
   
	    For a given finite set $W \subseteq L$ put 
   $$\partial W = \{w \in W | T(w) \not \subseteq W\}.$$
   In other words, $\partial W$ is the set of words $w \in W$ 
   for which at least one of the outputs of $T$ for $w$ is not in $W$.   
   Define the function  
   $\mathrm{F{\o}l}(\varepsilon) : (0,1) \rightarrow \mathbb{N}$ 
   as  
   $$\mathrm{F{\o}l} (\varepsilon) =  \min \{ \# W | 
   \#\partial W < \varepsilon \#W  \}.$$  
   It is assumed that the function $\mathrm{F{\o}l}(\varepsilon)$ 
   is defined on the  whole interval $(0,1)$, i.e.,  
   for every $\varepsilon \in (0,1)$ the set 
   $\{ W | \#\partial W < \varepsilon \#W  \}$
   is not empty.     
   \begin{itemize}
   	\item{We call the sequence
   		$f_n = \mathrm{F{\o}l} (\frac{1}{n}), n = 1, \dots, \infty$  
   		the F{\o}lner function of $T$.}
   \end{itemize}

	 Let $M$ be a finite multiset consisting of some words of $L$. 
   We denote by $T(M)$ the multiset obtained as follows. 
   Initially, $T(M)$ is empty. Then, for every word $w$ in $M$ 
   add the outputs of $T$ for $w$ to $T(M)$. 
   If $w$ has the multiplicity $m$ in $M$, 
   then this procedure must be  repeated $m$ times.
   Let $M_0$ be the multiset consisting of the word $w_0$ with the 
   multiplicity one. Put $M_1 = T(M_0)$ and, for $i>1$, put 
   $M_{i+1} = T(M_i)$. The total number of elements 
   (multiplicities are taken into account) in the multiset $M_n$ is $k^n$.
   Put $\ell_n$ to be 
   \begin{equation}
   \ell_n  = \frac{\sum \limits_{w \in M_n } m_w |w|}{k^n},
   \end{equation}
   where $m_w$ is the multiplicity of a word $w$ in $M_n$ and 
   $|w|$ is the length of  $w$. 
   In other words, $\ell_n$ is the average length of the words in the multiset $M_n$.   
   \begin{itemize}
   	\item{We call the sequence $\ell_n, n =1,\dots,\infty$ the average length growth function of
   		the pair $(T,w_0)$.}
   \end{itemize}  


  \section{Wreath Products of Groups: Basic Notation}
  \label{wreathproducts}

   Most of the labeled directed graphs in  this paper are obtained as Cayley graphs of wreath products of groups.  
   For the sake of convenience we describe basic notation 
   for  restricted wreath products 
   $A \mathrel{\wr} B$ in the present section. 
   For more details on wreath products see, e.g., \cite{KargapolovMerzljakov}.
   For given two groups $A$ and $B$,
   we denote by $A^{(B)}$ the set of all functions $f: B \rightarrow A$
   having finite supports.
   Recall that a function $f: B \rightarrow A$ 
   has finite support if the set  $\{ x \in B | f (x) \neq e \}$
   is finite, where $e$
   is the identity of $A$.
   Given $f \in A^{(B)}$ and $b \in B$, we define $f^b \in A^{(B)}$ as follows. Put $f^b (x) = f(bx)$ for all $x \in B$.
   The group $A \mathrel{\wr} B$ is the set product $A^{(B)} \times B$ 
   with the group multiplication given by 
   $( f, b ) \cdot ( f', b' )  =  ( f f'^{\, b^{-1}},  b b' )$.
   
   We denote by  $\mathfrak{i}_A$ the embedding  
	 $\mathfrak{i}_A : A \rightarrow A \mathrel{\wr} B$
	 for which $\mathfrak{i}_A : a \mapsto (f_a,e)$,
	 where $e$ is the identity of the group $B$ and 
	 $f_a \in A^{(B)}$ is the function  
	 $f_a: B \rightarrow A$ such that $f_a (e) = a$ and 
	 $f_a (x)$ is the identity of the group $A$ for 
	 every $x \neq e$.  	 
	 We denote by $\mathfrak{i}_B$ the embedding 
	 $\mathfrak{i}_B : B \rightarrow A \mathrel{\wr} B$  
   for which $\mathfrak{i}_B :  b \mapsto ({\bf e}, b )$, where
   ${\bf e}$ is the identity of the group 
   $A^{(B)}$; in other words, ${\bf e}$ 
	 is the function which maps all elements of $B$ to
   the identity of the group $A$.
   For the sake of convenience we will identify $A$ and $B$    
	 with the  subgroups 
	 $\mathfrak{i}_A (A) \leqslant A \mathrel{\wr} B$
	 and     
	 $\mathfrak{i}_B (B) \leqslant A \mathrel{\wr} B$, respectively.
   Let $S_A=\{a_1,\dots,a_n\} \subseteq A$ and 
   $S_B=\{b_1,\dots,b_m\} \subseteq B$ be 
   some sets of generators of the groups $A$ and $B$, respectively. 
   Then the set 
	 $S = \mathfrak{i}_A (S_A) \cup \mathfrak{i}_B (S_B)$ 
	 is a set of generators of $A \wr B$.
   The Cayley graph $\Gamma (A \wr B,S)$ can be obtained as follows. 
   The vertices of $\Gamma(A \wr B,S)$ are the elements of 
   $A \wr B$, i.e., all pairs $(f,b)$  such that 
   $f \in A^{(B)}$ and $b \in B$. 
   The right multiplication of an element $(f,b)$ by $a_i, i = 1, \dots,n$ is $(f,b) a_i = (\widehat{f}, b)$, where $\widehat{f}(s) = f(s)$ 
   if $s \neq b$ and $\widehat{f} (b) = f(b) a_i$. 
   The right multiplication of an element $(f,b)$ by $b_j, j = 1, \dots, m$
   is $(f,b) b_ j = (f, b b_j)$.
  
  \section{Asymptotic Behavior of the Numerical Characteristics}
  \label{sectionExamples}
	
	In this section we discuss asymptotic behavior of the numerical characteristics of Turing transducers of the class $\mathcal{T}$.   
	
	 \subsection{Growth Functions and F{\o}lner Functions}
   \label{subsectionfolner}	

        We first consider the behavior of  growth functions 
		for Turing transducers of the class $\mathcal{T}$. 
        
        
    Let $G$ be a  group with a finite set of generators 
		$Q \subseteq G$. Put $S=Q\cup Q^{-1}$. 
		We recall that the growth function of the pair $(G,Q)$ is the 
		sequence $b_n = \# B_n, n = 0, \dots, \infty$, 
		where $\# B_n$ is the number of elements 
		in the ball $B_n = \{g\in G | \ell_S (g) \leqslant n \}$.  
		Let $T \in \mathcal{T}$ be a Turing transducer translating 
		a language $L$ into  
	    $L' \subseteq L ^k$, where $k = \#S$.
	    Choose any word $w_0 \in L$.  
		The following claim is straightforward. 
		\begin{claim}
		   Suppose that the Cayley graph $\Gamma (G,S)$ 
		   is presented by  $T$.
		   Then the growth function $b_n$ of the pair 
		   $(T,w_0)$ coincides with the growth function of the pair $(G,Q)$. $\Box$ 
 		\end{claim}
		 One of the important questions in the group theory is whether
		 or not for a given pair $(G,Q)$  the growth series is rational. 
		 A similar question naturally arises for a pair $(T,w_0)$. 
		 It is easy to show an example of a pair 
		 $(T,w_0), T \in \mathcal{T}$ for which the growth series is not rational. 
		 \begin{example}
		 \label{H5example}
		 Stoll proved that the growth series 
		 of the Heisenberg group $H_5$ with respect to the 
		 standard set of generators is not rational~\cite{Stoll95}. 
		 The Cayley graph of $H_5$ is automatic
		 \cite[Example~6.7]{KKM11}.	
		   Therefore, we obtain that there exists a pair  
			 $(T,w_0), T\in \mathcal{T}$ for which the growth 
			 series $\sum b_n z^n$ is not rational. $\Box$ 
		\end{example}	
		Moreover,  a Turing transducer of the class 
		$\mathcal{T}$ may have a function $b_n, n= 0, \dots, \infty$ 
		of intermediate growth.		
		\begin{example}
		\label{Savchukexample}
		Miasnikov and Savchuk constructed an example 
		of a $4$--regular  automatic graph which has intermediate growth~\cite{MiasnikovSavchuk13}.  		
	     Therefore, we obtain that there exists a pair  
	   	 $(T,w_0),T \in \mathcal{T}$ 
		 for which the function $b_n,n=0,\dots,\infty$ has intermediate growth.$\Box$			 
	   \end{example}

	  We now consider the behavior of F{\o}lner function 
		$f_n, n = 1, \dots, \infty$ 
		for Turing transducers of the class $\mathcal{T}$. 
		F{\o}lner functions were considered by 
		A.~Vershik for Cayley graphs of amenable groups~\cite{VershikMir73}. 
	  We recall first some necessary definitions 
 	  regarding F{\o}lner functions  \cite{Erschlerisoperimetric03}. 
	  
		Let $G$ be an  amenable group  with a finite set 
	  of generators $Q \subseteq G$. 
	  Put $S = Q \cup Q^{-1}$. 
	  Let $E$ be the set of directed edges of $\Gamma(G,S)$. 
	  For a given finite set $U \subseteq G$ 
 	  the boundary $\partial U$ is defined as 
	  \begin{equation*}
	     \partial U = \{ u \in U | \exists v \in G [(u,v) \in E \wedge 
		 v \notin U] \}. 
  	  \end{equation*}
	  The function 
		$\mathrm{F{\o}l}_{G,Q}:(0,1) \rightarrow \mathbb{N}$ is defined as 
	  $$\mathrm{F{\o}l}_{G,Q} (\varepsilon) = 
		 \min \{\# U | \# \partial U < \varepsilon \# U\}.$$ 
		The F{\o}lner function 
		$\mathrm{F{\o}l}_{G,Q}: \mathbb{N} \rightarrow \mathbb{N}$ is defined 
		as $\mathrm{F{\o}l}_{G,Q}(n) = \mathrm{F{\o}l}_{G,Q} (\frac{1}{n})$. 	
	  The following claim is straightforward. 
	  \begin{claim}
	   Suppose that the Cayley graph $\Gamma(G,S)$ is presented by a Turing 
		 transducer $T \in \mathcal{T}$. Then for the F{\o}lner function $f_n$ of $T$, 
		 $f_n = \mathrm{F{\o}l}_{G,Q}(n)$. $\Box$
	  \end{claim}  	
    In this subsection we say that $f_1 (n) \sim f_2(n)$ if
	  there exists  $K \in \mathbb{N}$ such that 
	  $f_1 (Kn) \geqslant \frac{1}{K} f_2(n)$ and 
	  $f_2 (Kn) \geqslant \frac{1}{K} f_1(n)$, i.e., 
		$f_1(n)$ and $f_2(n)$ are equivalent up to a quasi--isometry. 
	  Let $Q' \subseteq G$ be another set generating $G$.
		Then $\mathrm{F{\o}l}_{G,Q} (n) \sim \mathrm{F{\o}l}_{G,Q'}(n)$. 
		In this subsection F{\o}lner 
		functions are considered up to quasi--isometries. 
		So, instead of $\mathrm{F{\o}l}_{G,Q} (n)$, 
		we will write $\mathrm{F{\o}l_G} (n)$.
	
		Let $G_1 = \mathbb{Z} \wr \mathbb{Z}$. Put  
		$G_{i+1} = G_{i} \wr \mathbb{Z}, i \geqslant 1$. 	
		It is shown \cite[Example~3]{Erschlerisoperimetric03} that 
		$\mathrm{F{\o}l}_{G_i} (n) \sim n^{(n^i)}$. 
		It follows from \cite[Theorem~3]{berdinskykhoussainov15}
		that for every integer $i \geqslant 1$ there exists a Turing 
		transducer $T_i \in \mathcal{T}$ for which a Cayley graph 
		of $G_i$ is presented by $T_i$. 
		The following theorem shows that the logarithm of 
		F{\o}lner functions for Turing transducers of the class $\mathcal{T}$
		can grow faster than any given polynomial.   
		\begin{theorem}
		\label{theorem_fn}   
		   For every integer $i \geqslant 1$
			 there exists a Turing transducer of the class 
		   $\mathcal{T}$ for which $f_n \sim n^{(n^i)}$.  $\Box$
		\end{theorem}	
		\begin{remark}
		Consider the group $\mathbb{Z} \wr (\mathbb{Z} \wr \mathbb{Z})$. 
		It is shown \cite[Example~4]{Erschlerisoperimetric03} that 
		$\mathrm{F{\o}l}_{\mathbb{Z} \wr (\mathbb{Z} \wr \mathbb{Z})} (n)
		 \sim n^{(n^n)}$. 
	    In particular, $\mathrm{F{\o}l}_{\mathbb{Z} 
		     \wr (\mathbb{Z} \wr \mathbb{Z})} (n) $ grows faster than 
		    $\mathrm{F{\o}l}_{G_i} (n)$ for every $i \geqslant 1$. 
		     However, it is not known whether or not there exists a Turing transducer
			 $T \in \mathcal{T}$ for which a Cayley graph of $\mathbb{Z} \wr (\mathbb{Z} \wr \mathbb{Z})$ is 
			 presented by $T$. $\Box$ 
		\end{remark}
	
  \subsection{Random Walk and Average Length Growth Functions}    
  \label{subsectionrandomwalk}  
  We recall first some necessary definitions \cite{Vershik99}.		
  Let $G$ be an infinite  group with 
  a set of generators $Q = \{ s_1, \dots, s_m \} \subseteq G$.
  Put $S = Q \cup Q^{-1} = 
       \{ s_1, \dots, s_m,$ $s_1 ^{-1}, \dots, s_m ^{-1} \}$.  
  For a given $g \in G$ we denote by $\ell_S (g)$ the minimal length of
  a word representing $g$ in terms of $S$.    
  We denote by $B_n$ the ball of the radius $n$, 
  $B_n = \{ g \in G | \ell_S (g) \leqslant n \}$.
  Let $\mu$ be a symmetric measure defined on $S$, i.e., 
  $\mu (s) = \mu (s^{-1})$ for all $s \in S$.
  The convolution 
  $\mu ^{*n} (g)$ on $B_n$ is defined as
  \begin{equation*}
  \mu ^{*n} (g) = \sum_{g = g_1 \dots g_n} \prod_{i = 1, \dots, n} \mu (g_i), 
  \end{equation*}
  where $g_i \in S$, $i=1,\dots,n$. 
  
  Let $c_n (g)$ be the number of words of length  
  $n$ over the alphabet $S$  representing the element $g \in G$.
  If $\mu$ is the uniform measure on $S$, then $\mu ^{*n} (g) = \frac{c_n(g)}{(2m)^n}$. Therefore, $\mu ^{*n} (g)$ is the probability
  that a $n$--step simple symmetric random walk on the Cayley 
  graph $\Gamma(G,S)$, which starts at the identity $e \in G$, 
  ends up at the vertex $g \in G$. 
  In this paper we consider only uniform measures $\mu$.   
  We denote by $E_{\mu^{*n}}[\ell_S]$  the average value of the 
  functional $\ell_S$ on the ball $B_{n}$ with respect to the measure 
  $\mu^{*n}$.  	 
  For some Cayley graphs of wreath products of groups
  we will show  asymptotic behavior of $E_{\mu^{*n}}[\ell_S]$ of the form 
  $E_{\mu^{*n}}[\ell_S] \asymp \ f(n)$, 
  where $g(n) \asymp f(n)$ means that  
  $\delta_1 f(n) \leqslant g(n) \leqslant \delta_2 f(n)$ for some 
  constants $\delta_2 \geqslant \delta_1 > 0$.

  Let $T \in \mathcal{T}$ be a Turing transducer translating a language $L$ into  $L'$.
  Suppose that the Cayley graph $\Gamma (G,S)$ is presented by  
  $T$. Let us choose any word $w_0 \in L$. 
  The Turing transducer $T$ provides the bijection 
  $\psi: L \rightarrow G$ such that $\psi^{-1}(e) = w_0$. 
  Therefore, we can consider 
  the average of the functional $|w|$ on the ball $B_n$ 
  with respect to the measure $\mu^{*n}$, where $|w|$ is the 
  length of a word $w \in L$. 
  The following claim is straightforward.  
  \begin{claim}
     For a $n$--step symmetric simple random walk on the Cayley graph 
     $\Gamma (G,S)$, $E_{\mu^{*n}}[|w|] = \ell_n$, where 
     $\ell_n$ is the $n$th element of the average length growth 
     function of the pair $(T,w_0)$. $\Box$   
  \end{claim} 
	The following proposition relates $\ell_n$ and $E_{\mu^{*n}} [\ell_S]$.
	\begin{proposition}	
		There exist constants $C_1$ and $C_2$ such that
		$\ell_n \leqslant C_1 E_{\mu^{*n}}[\ell_S] + C_2$ for all $n$.	
	\end{proposition}
	\begin{proof}
	   Recall that, by definition, there exists a constant $c$ such that for every input $x \in L$ and an output $y_j \in L, j = 1, \dots, 2m$, 
	   $|y_j| \leqslant |x| + c$. Put $C_1 = c$ and $C_2 = |w_0|$. 
	   Therefore, we obtain that the inequality  
	   $\ell_n \leqslant C_1 E_{\mu^{*n}}[\ell_S] + C_2$ holds
	   for all $n$. $\Box$
	\end{proof}	    
  It is easy to give examples of Turing transducers of the class 
  $\mathcal{T}$ for which
  $\ell_n \asymp \sqrt{n}$ and the growth function 
  $b_n$ is polynomial using  a unary--like representation 
  of integers. See Example \ref{example_ln_Zn} below.     
  \begin{example}
  \label{example_ln_Zn}
   Let $Q = \{s_1,\dots,s_m\}$ be the standard set of generators of the group $\mathbb{Z}^m$, where 
  $s_i = (\delta_{i}^1, \dots, \delta_{i} ^m)$ and 
  $\delta_{i}^j = 1$ if $i = j$, $\delta_{i} ^j =0$ if $i \neq j$. 
  Put $S = Q \cup Q^{-1}$. 
  It can be seen  that  
  there  exists a $(2m+1)$--tape Turing transducer $T \in \mathcal{T}$ 
	translating a language $L$ into a language $L' \subseteq L^{2m}$ for which 
  $\Gamma (\mathbb{Z}^m,S)$ is presented by $T$. 
  It is easy to see that a language $L$
  and an isomorphism between $\Gamma_T$ 
  and $\Gamma (\mathbb{Z}^m,S)$ can be chosen 
  in a way that $\ell_S (g)= |w|$, where $g \in \mathbb{Z}^m$ 
  and $w\in L$ is the word corresponding to $g$. In particular, 
  put the empty word $\epsilon$ to be the representative of the 
  identity $(0,\dots,0) \in \mathbb{Z}^m$.  
  Therefore, for such a Turing transducer 
  $T$, $\ell_n = E_{\mu^{*n}}[\ell_S]$. For a symmetric simple random walk on the $m$--dimensional grid, $E_{\mu^{*n}}[\ell_S] \asymp \sqrt{n}$. 
  For the proof see, e.g., \cite{Spitzerbook}.    
  So, for the pair $(T,\epsilon)$, $\ell_n \asymp \sqrt{n}$. The growth 
  function $b_n$ of $(T,\epsilon)$ is  polynomial. 
  Thus, we obtain   $(2m+1)$--tape Turing transducers $T_m, m = 1,\dots,\infty$ 
  for which $\ell_n \asymp \sqrt{n}$ and the growth function 
  $b_n$ is polynomial. $\Box$ 
  \end{example} 
   
    Is there 
    a Turing transducer of the class $\mathcal{T}$ for which $\ell_n \asymp \sqrt{n}$ and the growth function $b_n$ is exponential?    
    We will construct such a Turing transducer in Lemma \ref{lamplighterexampleln}.

    Let $H$ be a group with a set of generators $S_H =\{t_1,\dots,t_k\}$. 
	Consider the group $\mathbb{Z}_2 \wr H$. Let 
	$h \in \mathbb{Z}_2 ^{(H)}$ be the function 
	$h: H \rightarrow \mathbb{Z}_2$ such that 
	$h(g)=e$ if $g \neq e$ and $h(e)= a$, 
	where $a$ is the nontrivial element of $\mathbb{Z}_2$.     
    Let $Q = \{ t, t h, ht, hth | t \in S_H\}$ 
    be the set of generators of the  
	group $\mathbb{Z}_2 \wr H$. Put $S = Q \cup Q^{-1}$. 
    Consider a symmetric simple random walk on 
	the Cayley graph $\Gamma(\mathbb{Z}_2 \wr H, S)$. 
	It is easy to see that a $n$--step random 
	walk on $\Gamma(\mathbb{Z}_2 \wr H, S)$ corresponds to 
	a $n$--step random walk on $H$. 
	Put $P = S_H \cup S_H^{-1}$.
    Let $R_n $ be the number of different vertices visited 
	after walking $n$ steps on $\Gamma(H,P)$. We call  
	$R_n$ the range of a $n$--step random walk on $\Gamma(H,P)$.
	In the following proposition the   
    asymptotic behavior of $E_{\mu^{*n}}[\ell_S]$ is expressed in terms of $E_{\mu^{*n}}[R_n]$ -- the average range for a $n$--step  random walk on $\Gamma(H,P)$.
	\begin{proposition}
    \label{lemma2dyubina99}   
		 Let $H$ and $S$ be as above. For a symmetric simple random walk on 
		 $\Gamma(\mathbb{Z}_2 \wr H, S)$, 
		 $E_{\mu^{*n}}[\ell_S] \asymp E_{\mu^{*n}}[R_n]$.
    \end{proposition}
    \begin{proof}
	   For the proof see \cite[Lemma~2]{Dyubina99}. $\Box$
	\end{proof}
	
	\begin{lemma}	
	\label{lamplighterexampleln}
	   There exists a set of generators $S_1$ of the lamplighter  
	   group $\mathbb{Z}_2 \wr \mathbb{Z}$ 
	   for which the following statements hold. 
	   \begin{enumerate}[(a)]
	   	  \item For a simple symmetric random walk on 
	   	  	     $\Gamma(\mathbb{Z}_2 \wr \mathbb{Z},S_1)$, 
	   	  	     $E_{\mu^{*n}}[\ell_{S_1}] \asymp \sqrt{n}$.
	   	  \item There exists a Turing transducer $T_1 \in \mathcal{T}$ 
	   	  	     such that $\Gamma(\mathbb{Z}_2 \wr \mathbb{Z},S_1)$
	   	  	     is presented by $T_1$ and $\ell_n \asymp \sqrt{n}$.
	   \end{enumerate}	
	\end{lemma}
	\begin{proof}
	Let us consider  the lamplighter group $\mathbb{Z}_2 \wr \mathbb{Z}$. 
    Let $t$ be a generator of the subgroup 
    $\mathbb{Z} \leqslant \mathbb{Z}_2 \wr \mathbb{Z}$ and  
    $h \in \mathbb{Z}_2 ^{(\mathbb{Z})} \leqslant \mathbb{Z}_2 \wr \mathbb{Z}$ 
    be the function $h : \mathbb{Z} \rightarrow \mathbb{Z}_2$ 
    such that $h(z)= e$ if $z \neq 0$ and $h(0) = a$.
    Let $Q_1 = \{t,th,ht,hth\}$ be the set of generators of 
    $\mathbb{Z}_2 \wr \mathbb{Z}$ 
	and $S_1 = Q_1 \cup Q_1^{-1}$.
	For a simple symmetric random walk on $\Gamma(\mathbb{Z}, \{t,t^{-1}\})$,
    $E_{\mu^{*n}}[R_n] \sim \sqrt{n}$, 
    where $\sim$ here means asymptotic equivalence. 
    For the proof see, e.g., \cite{Spitzerbook}.
	Therefore, from Proposition \ref{lemma2dyubina99}  
	we obtain that for a simple symmetric random walk on 
	$\Gamma(\mathbb{Z}_2 \wr \mathbb{Z},S_1)$, 
	$E_{\mu^{*n}}[\ell_{S_1}] \asymp \sqrt{n}$.
    
    Let $Q_1' = \{t,h\}$ be a set of generators of 
    $\mathbb{Z}_2 \wr \mathbb{Z}$.  
    Put $S_1' = Q_1' \cup Q_1'^{-1} = \{t,t^{-1}, h\}$.
    In \cite[Theorem~2]{berdinskykhoussainov15} we 
	constructed an automatic presentation of the Cayley  
    graph $\Gamma (\mathbb{Z}_2 \wr \mathbb{Z}, S_1')$, 
    the bijection $\psi_1 : L_1 \rightarrow \mathbb{Z}_2 \wr \mathbb{Z}$, 
    for which the inequalities 
    $\frac{1}{3} \ell_{S_1'} (g) + \frac{2}{3} 
    \leqslant |w| \leqslant \ell_{S_1'} (g) + 1$ 
    hold for all $g \in \mathbb{Z}_2 \wr \mathbb{Z}$, 
    where $L_1$ is a regular language, $w = \psi_1^{-1}(g) \in L_1$ is 
    the word corresponding to $g$ and $|w|$ is the length of $w$.
    It is easy to see that
    $\frac{1}{2} \ell_{S_1} (g) \leqslant \ell_{S_1'} (g) 
	 \leqslant 3  \ell_{S_1} (g)$. 
    Therefore, we obtain that 
    $\frac{1}{6} \ell_{S_1} (g) + \frac{2}{3} \leqslant |w| \leqslant 
    3 \ell_{S_1} (g) + 1$ for all $g \in \mathbb{Z}_2 \wr \mathbb{Z}$. 
	  This implies that 
    $\frac{1}{6} E_{\mu^{*n}}[\ell_{S_1}] + \frac{2}{3} \leqslant  
    E_{\mu^{*n}}[|w|] \leqslant 3 E_{\mu^{*n}}[\ell_{S_1}] + 1$.       
    The bijection 
    $\psi_1 : L_1 \rightarrow \mathbb{Z}_2 \wr \mathbb{Z}$
    provides an automatic presentation 
    for the Cayley graph 
    $\Gamma (\mathbb{Z}_2 \wr \mathbb{Z}, S_1)$. 
    By Lemma \ref{fromautomatatotransprop1},
    we obtain that there exists 
    a $9$--tape Turing  transducer $T_1 \in \mathcal{T}$ 
    translating the language $L_1$ into some language $L_1' \subseteq L_1 ^8$ 
    for which 
    $\Gamma (\mathbb{Z}_2 \wr \mathbb{Z}, S_1)$ is presented by $T_1$.    
	  Therefore, we obtain that for  $T_1$,  
    $\ell_n \asymp \sqrt{n}$.   	  	
    Since the growth function of the group $\mathbb{Z}_2 \mathrel{\wr} \mathbb{Z}$ is exponentinal,     
		the growth function $b_n$ of $T_1$ is exponential. $\Box$ 
	\end{proof}     
      It is easy to give examples of  Turing transducers 
      of the class $\mathcal{T}$ for which $\ell_n \asymp n$
      and the growth function $b_n$ is exponential. 
      See Example \ref{example_ln_Fn} below.   
      \begin{example}
      	\label{example_ln_Fn}    
      	Let $F_m$ be the free group 
      	over $m$ generators $s_1,\dots,s_m$. 
      	Put $Q = \{s_1, \dots, s_m\}$ and  $S = Q \cup Q^{-1}$.  
      	There exists a natural automatic presentation of the Cayley
      	graph $\Gamma(F_m,S)$, the bijection $\psi: L \rightarrow F_m$, 
      	for which $L$ is the language of all reduced words over the alphabet $S$. 
      	In particular, the empty word $\epsilon$ represents the identity 
      	$e \in F_m$. 
      	The bijection $\psi$ maps a word $w \in L$ into the corresponding 
      	group element of $F_m$. It is clear that $\ell_S (g) = |w|$, where
      	$w = \psi^{-1} (g)$. For a symmetric simple random walk on 
      	$\Gamma (F_m,S) $, $E_{\mu^{*n}} [\ell_S] \asymp n$. 
      	Therefore, $E_{\mu^{*n}}[|w|] \asymp n$.  
      	Therefore, for each $m>1$ we obtain the pair 
      	$(T,\epsilon), T \in \mathcal{T}$ 
      	for which $\ell_n \asymp n$.  
				Since the growth function of the free group $F_m$ is exponential,  
				the growth function $b_n$ of the pair $(T,\epsilon)$ is exponential. $\Box$
      \end{example}     
   Is there a Turing transducer of the class $\mathcal{T}$ 
   for which $\ell_n$ grows between $\sqrt{n}$ and $n$? 
   We will answer on this question positively in Theorem \ref{theorem1ln}
   which follows from Proposition \ref{randomwalklamplighterprop2} below.

    Let $G$ be a group  with a set of generators $S_G  = \{g_1,\dots,g_m\}$. 
    Put $P = S_G \cup S_G ^{-1}$. 
    Assume that for a symmetric simple random walk on 
    $\Gamma(G,P)$, $\ell_n (\mu) \asymp n^\alpha$ 
    for some $0 < \alpha \leqslant 1$.	 	    
    Consider the wreath product $G \wr \mathbb{Z}$.
    Let $t$ be a generator of the subgroup 
    $\mathbb{Z} \leqslant G \wr \mathbb{Z}$. 
    Let $h_{i} \in G^{(\mathbb{Z})} \leqslant G \wr \mathbb{Z}$,  
    $i = 1, \dots, m$ 
    be the functions $h_{i}: \mathbb{Z} \rightarrow G$ 
    such that $h_{i} (z) = e$ if $z \neq 0$ and 
    $h_{i} (0) = g_i$. 
    Put $Q= \{h_{i}^p t h_{j}^q \, | \, i,j = 1,\dots,m;p,q=-1,0,1 \}$ 
    to be the set of generators of the group $G \wr \mathbb{Z}$ and
    $S = Q \cup Q^{-1}$. Consider a $n$--step random walk on 
    $\Gamma (G\wr  \mathbb{Z},S)$. 
    The following proposition shows asymptotic behavior of $E_{\mu^{*n}}[\ell_S]$. 
     
    \begin{proposition}
    \label{randomwalklamplighterprop2}	
  	Let $G$, $S$ and $\alpha$ be as above.
  	For a symmetric simple random walk on 
  	$\Gamma(G \wr \mathbb{Z},S)$,
  	$E_{\mu^{*n}}[\ell_S] \asymp n^\frac{1+\alpha}{2}$.			
    \end{proposition}
    \begin{proof}
       For the proof see \cite[Lemma~3]{Erschler04}. $\Box$
    \end{proof}
  
  	\begin{theorem}
  	\label{theorem1ln}   
  	   For every $\alpha < 1$ there exists a Turing transducer  
  	   $T \in \mathcal{T}$ for which $\ell_n \asymp n^\beta$ 
  	   for some $\beta$ such that $\alpha < \beta < 1$ and 
  	   the growth function $b_n$ is exponential.  	   
  	\end{theorem}	
   \begin{proof}
   Let us consider the sequence of wreath products $G_m, m = 1,\dots,\infty$
  such that $G_1 = \mathbb{Z}_2 \wr \mathbb{Z}$ and 
  $G_{m+1} = G_m \wr \mathbb{Z}$, $m \geqslant 1$. 
  From Lemma \ref{lamplighterexampleln}~$(a)$  and Proposition \ref{randomwalklamplighterprop2} 
  we obtain that for every $m > 1$ there exists a proper set of
  generators $Q_m \subseteq G_m$ such that for a symmetric simple 
	random walk on the Cayley graph $\Gamma (G_m, S_m)$, 
  $E_{\mu^{*n}}[\ell_{S_m}] \asymp n^{1 - \frac{1}{2^m}}$, 
  where $S_m = Q_m \cup Q_m ^{-1}$.    	  
  It follows from \cite[Theorem~3]{berdinskykhoussainov15} 
  that for every $m>1$ there is an automatic presentation 
  of the Cayley graph $\Gamma (G_m,S_m')$, the bijection $\psi_m: L_m \rightarrow G_m$, 
  for which the inequalities 
  $\delta_1' \ell_{S_m'} (g) + \lambda_1'  \leqslant |w| 
   \leqslant \delta_2' \ell_{S_m'} (g)  + \lambda_2'$  
  hold for all $g \in G_m$ for some constants
  $\delta_2' > \delta_1' >0, \lambda_1', \lambda_2'$, 
  where $L_m$ is a regular language and
  $S_m' = Q_m' \cup {Q_m'}^{-1}$ for some proper set of generators 
  $Q_m' \subseteq G_m$, and 
  $w = \psi_m ^{-1}(g)$ is the word representing $g$. 
  Therefore, the inequalities 
  $\delta_1 \ell_{S_m} (g) + \lambda_1 \leqslant |w| 
	 \leqslant \delta_2 \ell_{S_m} (g) + \lambda_2$ 
  hold for all $g\in G_m$ for 
  some constants $\delta_2>\delta_1>0,\lambda_1,\lambda_2$.
  This implies that 
  $\delta_1 E_{\mu^{*n}}[\ell_{S_m}] + \lambda_1 \leqslant E_{\mu^{*n}}[|w|] \leqslant \delta_2 E_{\mu^{*n}} [\ell_{S_m}] + \lambda_2$. 
  Therefore, 
  $E_{\mu^{*n}}[|w|] \asymp n^{1 - \frac{1}{2^m}}$.
  
  For every $m>1$ the bijection $\psi_m : L_m \rightarrow G_m$ 
  provides an automatic presentation of the Cayley graph 
  $\Gamma (G_m, S_m)$. It follows from Lemma 
  \ref{fromautomatatotransprop1} that there is a 
  $(k_m+1)$--tape Turing transducer $T_m \in \mathcal{T}$ 
  translating the language $L_m$ into 
  $L_m' \subseteq L_m ^{k_m}$ for which, after proper relabeling,
	$\Gamma_{T_m} \cong \Gamma(G_m,S_m)$. The numbers 
	$k_m, m = 1, \dots, \infty$ can be
  obtained recurrently as follows.   
  It is easy to see that $k_{m+1} = 2  (k_m + 1)^2$ for $m \geqslant 1$  
  and $k_1 = 8$, which is simply the number 
  of elements in $S_1$ (see Lemma \ref{lamplighterexampleln}). 
	  So, we obtain that for  $T_m, m>1$, 
    $\ell_n \asymp n^{1-\frac{1}{2^m}}$. 
		For every $m>1$, since the growth function of the group $G_m$ is exponential, 
		the growth function $b_n$ of $T_m$ is exponential. $\Box$ 
  \end{proof}

	
\section{Discussion}
\label{discussionSection}

     In this paper we addressed the problem of finding 
     characterizations of Cayley automatic groups.
     Our approach was to define and then study three numerical 
     characteristics of Turing transducers of the special class 
     $\mathcal{T}$.  
     This class of Turing transducers was obtained from automatic 
		 presentations of labeled directed graphs.      
     The numerical characteristics that we defined are the analogs 
		 of growth functions, F{\o}lner functions and drifts of simple random walks  
		 for Cayley graphs of groups.    
     We hope that further study of asymptotic behavior of 
     these three numerical characteristics of Turing transducers
     of the class $\mathcal{T}$ will yield some characterizations 
     for Cayley automatic groups.

	 Two open questions are apparent from the results of Section \ref{sectionExamples}.     
     
	   \begin{itemize}
	   	\item{Theorem \ref{theorem_fn} shows 
			that for every integer $i \geqslant 1$
			there exists a Turing transducer of the class $\mathcal{T}$ 
	 		for which $f_n \sim n^{(n^i)}$. 
			Is there a Turing transducer 
	  	$T \in \mathcal{T}$ for which the F{\o}lner function 	    
	   	grows faster than  $n^{(n^i)}$ for all $i \geqslant 1$?} 	 
	   \end{itemize}

     \begin{itemize}
       \item{Theorem \ref{theorem1ln} tells us that for  every $\alpha < 1$ there exists a Turing transducer  $T \in \mathcal{T}$ for which $\ell_n \asymp n^\beta$ 
       	for some $\beta$ such that $\alpha < \beta < 1$. Is there a Turing transducer 
       	$T \in \mathcal{T}$ for which $\ell_n$ grows faster than  $n^\alpha$ for every $\alpha<1$  but slower than $n$?} 
     \end{itemize}	     

\section*{Acknowledgments}
The author thanks Bakhadyr Khoussainov and the anonymous reviewers for 
useful suggestions. The author thanks Sunny Daniels for proofreading 
a draft of this paper and making several changes to it.

\bibliographystyle{splncs03}
    
\bibliography{bibliography}

\begin{thebibliography}{10}
\providecommand{\url}[1]{\texttt{#1}}
\providecommand{\urlprefix}{URL }

\bibitem{dlt14}
Berdinsky, D., Khoussainov, B.: On automatic transitive graphs. In: Shur, A.,
  Volkov, M. (eds.) Developments in Language Theory 2014, Lecture Notes in
  Computer Science, vol. 8633, pp. 1--12. Springer Berlin Heidelberg (2014)

\bibitem{berdinskykhoussainov15}
Berdinsky, D., Khoussainov, B.: Cayley automatic representations of wreath
  products. International Journal of Foundations of Computer Sceince  27(2),
  147 -- 159 (2016)

\bibitem{BlumensathGradel04}
Blumensath, A., Gr\"adel, E.: Finite presentations of infinite structures:
  automata and interpretations. Theory of Computing Systems  37,  641--674
  (2004)

\bibitem{BlumensathGradel00}
Blumensath, A., Grädel, E.: Automatic structures. In: Proceedings of the
  Fifteenth Annual IEEE Symposium on Logic in Computer Science (LICS 2000). pp.
  51--62. IEEE Computer Society Press (June 2000)

\bibitem{Stephan_lmcs_13}
Case, J., Jain, S., Seah, S., Stephan, F.: Automatic functions, linear time and
  learning. Logical Methods in Computer Science  9(3:19),  1--26 (2013)

\bibitem{Dyubina99}
Dyubina, A.: An example of the rate of growth for a random walk on a group.
  Russian Mathematical Surveys  54(5),  1023--1024 (1999)

\bibitem{Epsteinbook}
Epstein, D.B.A., Cannon, J.W., Holt, D.F., Levy, S.V.F., Paterson, M.S.,
  Thurston, W.P.: Word Processing in Groups. Jones and Barlett Publishers.
  Boston, MA (1992)

\bibitem{Erschlerisoperimetric03}
Erschler, A.: On {I}soperimetric {P}rofiles of {F}initely {G}enerated {G}roups.
  Geometriae Dedicata  100(1),  157--171 (2003)

\bibitem{Erschler04}
Erschler, A.: On the asymptotics of drift. Journal of Mathematical Sciences
  121(3),  2437--2440 (2004)

\bibitem{Hodgson83}
Hodgson, B.R.: D\'ecidabilit\'e par automate fini. Annales des sciences
  math\'ematiques du Qu\'ebec  7(1),  39--57 (1983)

\bibitem{KargapolovMerzljakov}
Kargapolov, M.I., Merzljakov, J.I.: Fundamentals of the theory of groups.
  Springer--Verlag New York Inc. (1979)

\bibitem{KKM11}
Kharlampovich, O., Khoussainov, B., Miasnikov, A.: From automatic structures to
  automatic groups. Groups, Geometry, and Dynamics  8(1),  157--198 (2014)

\bibitem{khoussainov2007three}
Khoussainov, B., Minnes, M.: Three lectures on automatic structures.
  Proceedings of Logic Colloquium pp. 132--176 (2007)

\bibitem{KhoussainovNerode95}
Khoussainov, B., Nerode, A.: Automatic presentations of structures. In:
  Leivant, D. (ed.) Logic and Computational Complexity, Lecture Notes in
  Computer Science, vol. 960, pp. 367--392. Springer Berlin Heidelberg (1995)

\bibitem{khoussainov2008open}
Khoussainov, B., Nerode, A.: Open questions in the theory of automatic
  structures. Bulletin of the EATCS  94,  181--204 (2008)

\bibitem{Medunabook}
Meduna, A.: Automata and {L}anguages. {T}heory and {A}pplications.
  Springer-Verlag London Ltd. (2000)

\bibitem{MiasnikovSavchuk13}
Miasnikov, A., Savchuk, D.: An example of an automatic graph of intermediate
  growth. Ann. Pure Appl. Logic  166(10),  1037--1048 (2015)

\bibitem{OliverThomas05}
Oliver, G.P., Thomas, R.M.: Automatic presentations for finitely generated
  groups. In: Diekert, V., Durand, B. (eds.) STACS 2005, Lecture Notes in
  Computer Science, vol. 3404, pp. 693--704. Springer Berlin Heidelberg (2005)

\bibitem{Spitzerbook}
Spitzer, F.: Principles of random walk. Van Nostrand, Princeton (1964)

\bibitem{Stoll95}
Stoll, M.: Rational and transcendental growth series for the higher
  {H}eisenberg groups. Invent. Math.  126,  85--109 (1996)

\bibitem{VershikMir73}
Vershik, A.: Countable groups that are close to finite ones, {A}ppendix in {F}.
  {P}. {G}reenleaf, {I}nvariant {M}eans on {T}opological {G}roups and their
  {A}pplications, {M}oscow, {M}ir, 1973 (in {R}ussian), a revised {E}nglish
  translation: {A}menability and approximation of infinite groups. Selecta
  Math.  2(4),  311--330 (1982)

\bibitem{Vershik99}
Vershik, A.: Numerical characteristics of groups and corresponding relations.
  Journal of Mathematical Sciences  107(5),  4147--4156 (2001)

\end{thebibliography}

\end{document}